\documentclass[11pt,reqno]{amsart}

\usepackage{lineno,hyperref}

\modulolinenumbers[5]

\usepackage{amsmath}
\usepackage{amssymb}
\usepackage{mathrsfs}
\usepackage{amsthm}
\usepackage{bm}
\usepackage{graphicx}
\usepackage{booktabs}
\usepackage{float}
\usepackage{subfig}
\usepackage{geometry}
\usepackage{color}

\newtheorem{Def}{Definition}[section]
\newtheorem{lem}[Def]{Lemma}
\newtheorem{theo}[Def]{Theorem}
\newtheorem{pro}[Def]{Proposition}
\newtheorem{rem}[Def]{Remark}

\newtheorem{assum}{Assumption}
\newtheorem{cor}[Def]{Corollary}

\newcommand{\LL}{\langle}
\newcommand{\RR}{\rangle}

\newcommand{\mcal}{\mathcal}
\newcommand{\DW}{\Delta W}
\newcommand{\mscr}{\mathscr}
\newcommand{\mbb}{\mathbb}
\newcommand{\mbf}{\mathbf}

\newcommand{\ud}{\mathrm d}

\numberwithin{equation}{section}
\allowdisplaybreaks[4]

\begin{document}

\title[LDP]{Convergence analysis for minimum action methods coupled with a finite difference method}

\author{Jialin Hong}
\address{Academy of Mathematics and Systems Science, Chinese Academy of Sciences, Beijing
100190, China; School of Mathematical Sciences, University of Chinese Academy of	Sciences, Beijing 100049, China}
\email{hjl@lsec.cc.ac.cn}

\author{Diancong Jin}
\address{School of Mathematics and Statistics, Huazhong University of Science and Technology, Wuhan 430074, China;
	Hubei Key Laboratory of Engineering Modeling and Scientific Computing, Huazhong University of Science and Technology, Wuhan 430074, China.}
\email{jindc@hust.edu.cn (Corresponding author)}

\author{Derui Sheng}
\address{Department of Applied Mathematics, The Hong Kong Polytechnic University, Hung Hom, Kowloon, Hong Kong.}
\email{dsheng@polyu.edu.hk}

\thanks{This work is funded by the National key R\&D Program of China under Grant No. 2020YFA0713701, National Natural Science Foundation of China (Nos. 11971470, 12031020,  11871068, 12201228 and 12171047), and the Fundamental Research Funds for the Central Universities 3004011142.}


\begin{abstract}
The minimum action method (MAM) is an effective approach to numerically solving  minimums and minimizers of Freidlin--Wentzell (F-W) action functionals, which is used to study 
the most probable transition path and probability of the  occurrence of transitions  for stochastic differential equations (SDEs) with small noise. In this paper, we focus on  MAMs based on a finite difference method, and present the convergence analysis of minimums and minimizers of the discrete F-W action functional. The main result shows that the convergence orders of the minimum of the discrete F-W action functional in the cases of multiplicative noises and additive noises are $1/2$ and $1$, respectively. Our main result also reveals the convergence of the stochastic $\theta$-method for SDEs with small noise in terms of large deviations. 
\end{abstract}

\maketitle


\section{Introduction}
\label{Sec1}
Dynamical systems are often  perturbed by various  environmental noises. Although   the amplitude of random perturbations may be quite small, it can considerably impact the dynamics of underlying systems. For example, the transition between stable equilibrium points of the deterministic systems could take place when the small noise is introduced, which is impossible for the deterministic case. These transition events are rare but play important roles in many processes, such as nucleation events, chemical reactions,  regime change in climate and so on; see, e.g., \cite{Weinan2002,Kampen81,YaoRen2015}. As a general model  perturbed by small noise, we consider in this paper the following nonlinear stochastic differential equation (SDE) with multiplicative noise:
\begin{align}
	\label{SDE}
	\ud X^{\epsilon}(t)=b(X^\epsilon(t))\ud t+\sqrt{\epsilon}\sigma(X^\epsilon(t))\ud W(t),\quad t>0,
\end{align}
Here, $b:\mbb R^d\to\mbb R^d$ and $\sigma:\mbb R^d\to\mbb R^{d\times d}$ satisfy Assumption \ref{Assum1}. Moreover, $\epsilon>0$ denotes the noise intensity and is sufficiently small, and $\{W(t),t\ge0\}$ is a standard $d$-dimensional Brownian motion defined on a complete filtered probability space $(\Omega,\mscr F,\left\{\mscr F_t\right\}_{t\ge 0},\mbf P)$, with  $\left\{\mscr F_t\right\}_{t\ge0}$ satisfying the usual conditions. 

For \eqref{SDE}, a basic problem is to understand the transitions of $X^\epsilon$ between different states. For this end, one usually resorts to Freidlin--Wentzell (F-W) theory of large deviations. The F-W theory gives a rigorous estimate of the distribution of the trajectory $X^\epsilon(\omega,\cdot)$, which asserts that for any  $\varphi\in\mbf C([0,T];\mbb R^d)$ and $0<\delta\ll 1$,
$$\mbf P(\sup_{t\in[0,T]}|X^\epsilon(t)-\varphi(t)|\le \delta)\asymp \exp\Big(-\frac{1}{\epsilon} S_T(\varphi)\Big),\quad\text{as}~\epsilon\to 0.$$ 
Here, $S_T$ is called the F-W action functional, or large deviations rate function in some texts, which is given by
\begin{align}\label{ST}
	S_T(\varphi)=\frac{1}{2}\int_0^T|\sigma^{-1}(\varphi(t))(\varphi'(t)-b(\varphi(t)))|^2\ud t,\quad\varphi\in\mbf H^1(0,T;\mbb R^d).
\end{align}
Especially,  
the F-W theory shows that for any $x_0,x\in\mbb R^d$,
\begin{align}
	\lim\limits_{\epsilon\to 0}\epsilon\ln\mbf P(X^\epsilon\in\mbf C([0,T];\mbb R^d): X^\epsilon(0)=x_0,~X^\epsilon(T)=x)=-\inf_{\substack{\varphi(0)=x_0,\\\varphi(T)=x}}S_T(\varphi).
\end{align}
This implies that the most probable transition path, connecting $x_0$ and $x$ over the time interval $[0,T]$, is the minimizer of the action functional $S_T$. A central question in the F-W theory is how to compute the minimum and minimizer of $S_T$, i.e., how to address the following minimization problem:
\begin{align*}
	\text{Probem I}:\quad S_T(\varphi*)=\inf_{\substack{\varphi(0)=x_0,\\\varphi(T)=x}}S_T(\varphi).
\end{align*}
In addition, if one is interested in the case where $T$ is relaxed, the following minimization problem occurs:
\begin{align*}
	\text{Probem II}:\quad S_{T^*}(\varphi*)=\inf_{T>0}\inf_{\substack{\varphi(0)=x_0,\\\varphi(T)=x}}S_T(\varphi).
\end{align*}
The minimums and minimizers of Problems I and II measure the difficulty of $X^\epsilon$ transferring from $x_0$ to the vicinity of $x$. And the minimizer $\varphi^*$ of Problem I or Problem II corresponds to the most probable transition path connecting $x_0$ and $x$, which is  called the minimum action path (MAP). The numerical methods, which approximately solve Problem I or Problem II, are generally called the minimum action methods (MAMs).  

The MAM is first proposed in \cite{Weinan2004}, and has derived many variants. Here, we only refer to some of them without being exhaustive. For gradient systems (i.e., $b=-\nabla V$ for some potential $V$ in \eqref{SDE}), whose MAP is parallel to the drift term, the MAM 
includes the string method \cite{Weinan2002String}, the nudged elastic band method \cite{MAM1998}, etc. For nongradient systems, a numerical difficulty in finding the MAP lies in that the transition path spends most of its  time around critical points with slow dynamics.   When the time interval $[0,T]$ is discretized uniformly, most points along the numerical transition path will cluster  around the critical points due to the slow dynamics there. Thus, the MAP, mainly determined by fast dynamics, is only captured by a few grid points (see, e.g., \cite{WangX2018}). Some methods have been developed to overcome the above difficulty, such as the   geometric MAM (gMAM) \cite{gMAM2008}, adaptive MAM (aMAM) \cite{WangX2018,aMAM2008} and the MAM with optimal linear time scaling (tMAM) \cite{WangXL2015}.

From the practical point of view, there have been quite many algorithms based on the MAM, which are well developed to numerically solve Problems I and II. However, the rigorous numerical analyses, especially the convergence analysis for minimums and minimizers of discrete F-W action functionals, are very few. We are only aware of that authors in \cite{ZhaiJ2018} analyze the 
convergence  for a conforming finite element approximation of MAMs for the case of additive noises,  through the theory of $\Gamma$-convergence. In addition, they obtain the convergence rate for Problem II in the case that $\sigma=I_d$  and $b$ is linear. Besides the finite element method, the finite difference method (FDM) is also often used to discrete F-W action functionals when MAMs are applied; see, e.g., \cite{Weinan2004,ZhouX2010,aMAM2008}. But to the best of our knowledge, there is no any convergence analysis
for MAMs based on FDMs,  which motivates us  to  develop the corresponding theoretical analysis.

In this paper, we focus on the MAM, coupled with an FDM, for Problem I.
For $N\in\mbb N^+$, let $\left\{t_0<t_1<\cdots<t_{N-1}<t_N\right\}$ be a uniform partition of $[0,T]$ with $t_n=nh$, $n=0,1,\ldots,N$, where $h=\frac{T}{N}$ is the step-size. Then, we use the following FDM to discrete $S_T$: 
\begin{align}\label{STh}
	S_{T,h}(\psi_1,\psi_2,\ldots,\psi_{N-1})=\frac{h}{2}\sum_{n=0}^{N-1}\Big\vert \sigma^{-1}(\psi_n)\Big[\frac{\psi_{n+1}-\psi_n}{h}-b((1-\theta)\psi_n+\theta\psi_{n+1})\Big]\Big\vert ^2
\end{align}
with the parameter $\theta\in[0,1]$ and the constraints $\psi_0=x_0$ and $\psi_N=x$. Accordingly, we obtain a numerical discretization of Problem I: 
$$\text{Problem III}:~ S_{T,h}(\psi_1^*,\psi_2^*,\ldots,\psi_{N-1}^*)=\inf_{(\psi_1,\psi_2,\ldots,\psi_{N-1})\in\mbb R^{N-1}}S_{T,h}(\psi_1,\psi_2,\ldots,\psi_{N-1}).$$
One difficulty in proving the convergence of the minimum of $S_{T,h}$ is  that the feasible region of Problem III is not a subset of  $\mbf H^1(0,T;\mbb R^d)$.  This is different from \cite{ZhaiJ2018} where the conforming finite element method is used to discrete $S_T$, which means that the feasible region of the discrete version for Problem I is naturally embedded into $\mbf H^1(0,T;\mbb R^d)$. In order to overcome this difficulty, we prove that Problem III is equivalent to $$\text{Problem IV}:\quad\hat{S}_{T,h}(\varphi^*_h)=\inf\limits_{\{\varphi\in\mbf H^1(0,T;\mbb R^d):\varphi(0)=x_0,~\varphi(T)=x\}}\hat{S}_{T,h}(\varphi)$$
with
\begin{align}\label{hatSTh}
	\hat{S}_{T,h}(\varphi):=\frac{1}{2}\int_{0}^{T}\left\vert \sigma^{-1}(\varphi(\hat t))\left(\varphi'(t)-b\left((1-\theta)\varphi(\hat t)+\theta\varphi(\check t)\right)\right)\right\vert ^2\ud t,	
\end{align}
where $\hat t:=\max\left(\left\{t_0,t_1,\ldots,t_N\right\}\cap[0,t]\right),~ \check{t}:=\min\left(\left\{t_0,t_1,\ldots,t_N\right\}\cap[t,T]\right)$ for each $t\in[0,T]$.
The equivalence between Problem III and Problem IV enables us to study the error between minimums of $S_T$ and $\hat{S}_{T,h}$ endowed with same constrained space $\mbf H^1(0,T;\mbb R^d)$.  
Our strategy for the convergence analysis  is through the minimizer sequence of $\{\hat{S}_{T,h}\}_{h>0}$, which
relies on two key estimates: one is the equi-coerciveness of  $\hat{S}_{T,h}$ in Lemma \ref{SThcoer}, i.e., the exponential upper bound  of $\mbf H^1$-elements in terms of   $\hat{S}_{T,h}$; the other is the uniform error estimate between $S_{T}$ and $\hat{S}_{T,h}$ on any given bounded set (see Lemma \ref{localerror}).
Our  main result (Theorem \ref{convergence}) shows that the minimum of $\hat{S}_{T,h}$ converges to that of $S_T$, where the convergence orders 
in the cases of multiplicative noises and additive noises are $1/2$ and $1$, respectively. On basis of the convergence of the minimum of $\hat{S}_{T,h}$, we also establish the convergence of any minimizer sequence of $\{\hat{S}_{T,h}\}_{h>0}$ in Theorem \ref{convergence2}. We would like to mention that compared with the theory of $\Gamma$-convergence, our approach can provide the convergence order of the minimumn of the discrete F-W action functional. 
In addition, we show that the minimums of $S_T$ and $\hat{S}_{T,h}$ are the large deviations rate functions (LDRFs) of $\{X^{\epsilon}(T)\}_{\epsilon>0}$ and $\{X^{\epsilon}_N\}_{\epsilon>0}$, respectively, where $X^{\epsilon}_N\approx X^{\epsilon}(T)$ is the numerical solution  generated by the stochastic $\theta$-method for \eqref{SDE}. Thus, our main result also reveals the convergence of the stochastic $\theta$-method for SDEs with small noise in terms of large deviations.

The rest of this paper is organized as follows. Section \ref{Sec2} gives the existence of minimizers of  $S_T$ and $\hat{S}_{T,h}$. Section \ref{Sec3} presents the convergence analysis of minimums and minimizers of $\hat{S}_{T,h}$. As an application of our main result, we analyze the convergence of the LDRF of stochastic $\theta$-method in Section \ref{Sec4}. 
Finally, Section \ref{Sec6} recalls our main results and proposes some future aspects.

\section{Existence of minimizers of action functionals}\label{Sec2}
In this section, we present the existence of minimizers of both Problems I and IV. We begin with some notations. Throughout this paper, let $a\wedge b$ denote the minimum of $a$ and $b$ for any $a,b\in\mbb R$. Let $\mbb N^+$ be the set of all positive integers. Denote by $\vert \cdot\vert $ the $2$-norm of a vector or matrix, and $\langle\cdot,\cdot\rangle$ denotes the inner product of vectors.
For  $T\in(0,+\infty)$ and  $d\in\mbb N^+$, denote by $\mbf C\left([0,T],\mbb R^d\right)$ the space of all continuous functions $f:[0,T]\to\mbb R^d$, equipped with the supremum norm $\|f\|_0=\sup_{t\in[0,T]}\vert f(t)\vert $. And for given $x\in\mbb R^d$, denote $\mbf C_{x}\left([0,T],\mbb R^d\right):=\left\{f\in\mbf C\left([0,T],\mbb R^d\right): f(0)=x\right\}$.
Let $\mbf L^2(0,T;\mbb R^d)$  stand for the space of all square integrable functions with the inner product $\langle f,g\rangle_{\mbf L^2}=\int_{0}^T\langle f(t),g(t)\rangle\ud t$ and the induced norm $\|f\|_{\mbf L^2}:=\sqrt{\langle f,f\rangle_{\mbf L^2}}$ for any $f,g\in \mbf L^2(0,T;\mbb R^d)$.
Denote $\mbf H^1(0,T;\mbb R^d):=\big\{f:[0,T]\to\mbb R^d:f~\text{is absolutely continuous}~\text{and}~f'\in\mbf L^2(0,T;\mbb R^d)  \big\}$, endowed  with norm $\|f\|_{\mbf H^1}:=\|f\|_{\mbf L^2}+\|f'\|_{\mbf L^2}$. Also
for given $x_0,\,x\in\mbb R^d$, denote $\mbf H^1_{x_0}(0,T;\mbb R^d):=\{f\in \mbf H^1(0,T;\mbb R^d): f(0)=x_0\}$ and $\mbf H^1_{x_0,x}(0,T;\mbb R^d):=\{f\in \mbf H^1(0,T;\mbb R^d): f(0)=x_0,~f(T)=x\}$. In addition, let $\mbf{W}^{m,p}(0,T;\mbb R^d)$ ($m\in\mbb N^+$, $1\le p< \infty$), endowed with the norm 
$\|f\|_{\mbf{W}^{m,p}}=(\sum_{k=0}^m\int_0^T\vert f^{(k)}(t)\vert ^p\ud t)^{\frac1p}$,
denote the usual Sobolev space consisting of all $\mbf L^p$-integrable functions whose weak derivatives up to $m$ order are also $\mbf L^p$-integrable.

\subsection{Existence of minimizer of $S_T$}
In this subsection, we show that $S_T$ admits a minimizer by the coerciveness and weakly lower semicontinuity of $S_T$. Throughout this paper, we always let the following conditions hold without extra statements.
\begin{assum}\label{Assum1}
	$b$ and $\sigma$ are globally Lipschitz continuous, i.e.,  there is some constant $L>0$ such that
	\begin{align}\label{lip}
		\vert b(x)-b(y)\vert +\vert \sigma(x)-\sigma(y)\vert \leq L\vert x-y\vert \quad \forall~x,y\in\mbb R^d.
	\end{align}
	In addition, $\sigma(x)$ is invertible for each $x\in\mbb R^d$.
\end{assum}
\noindent It follows from \eqref{lip} that both $b$ and $\sigma$ grow at most linearly. For convenience, we also assume that $\vert b(0)\vert +\vert \sigma(0)\vert \leq L$ so that 
\begin{align}\label{linear}
	\vert b(x)\vert +\vert \sigma(x)\vert \leq L(1+\vert x\vert )\quad\forall~ x\in\mbb R^d.
\end{align}

In the later arguments, we will often use the following facts.
\begin{pro}\label{Prop}
	The   following properties hold.
	\begin{itemize}
		\item[(1)]  $\sigma^{-1}$ is locally Lipschitz continuous, i.e., for each $R>0$, there exists some constant $L_R>0$ such that for any $x,y\in\mbb R^d$ with $|x|\le R$ and $|y|\le R$,
		\begin{align*}
			\vert \sigma^{-1}(x)-\sigma^{-1}(y)\vert \leq L_{R}\vert x-y\vert.
		\end{align*}
		\item[(2)] For any $\varphi\in\mbf H^1_{x_0}(0,T;\mbb R^d)$, 
		\begin{align*}
			\vert \varphi(t)-\varphi(s)\vert &\leq(t-s)^{1/2}\left(\int_{s}^{t}\vert \varphi'(r)\vert ^2\ud r\right)^{1/2}\quad \forall~0\leq s\leq t\leq T,\\
			\|\varphi\|_0&\leq \vert x_0\vert +\sqrt{T}\|\varphi\|_{\mbf H^1}.
		\end{align*}
	\end{itemize}
\end{pro}

The coerciveness and weakly lower semicontinuity of $S_T$ are established in Proposition \ref{STcoercive} and Proposition \ref{STlower}, respectively. 
\begin{pro}\label{STcoercive}There exists some constant $C_0>0$ such that for any $\varphi\in\mbf H^1_{x_0,x}(0,T;\mbb R^d)$,
	$\|\varphi\|_{\mbf H^1}\leq C_0e^{C_0S_T(\varphi)}.$
\end{pro}
\begin{proof}
	We use $K(x_0,T,L)$ to denote some positive constant depending on $x_0,T$ and $L$, which may vary for each appearance. 
	Denote $f(t)=\sigma^{-1}(\varphi(t))\left(\varphi'(t)-b(\varphi(t))\right)$  for a.e. $t\in[0,T]$. Then $\|f\|_{\mbf L^2}^2=2S_T(\varphi)$ and 
	$\varphi(t)=x_0+\int_{0}^tb(\varphi(s))\ud s+\int_{0}^t\sigma(\varphi(s))f(s)\ud s$.
	By the H\"older inequality and \eqref{linear}, for each $t\in[0,T]$,
	\begin{align*}
		\vert \varphi(t)\vert ^2&\leq 3\vert x_0\vert ^2+3t\int_0^t\vert b(\varphi(s))\vert ^2\ud s+3\int_{0}^t\vert \sigma(\varphi(s))\vert ^2\ud s\int_0^t\vert f(s)\vert ^2\ud s\\
		&\leq 3\vert x_0\vert ^2+6TL^2\int_0^t\left(1+\vert \varphi(s)\vert ^2\right)\ud s+12L^2S_T(\varphi)\int_0^t\left(1+\vert \varphi(s)\vert ^2\right)\ud s\\
		&\leq K(x_0,T,L)\left(1+S_T(\varphi)\right)+K(x_0,T,L)\left(1+S_T(\varphi)\right)\int_0^t\vert \varphi(s)\vert ^2\ud s.
	\end{align*}
	According to the Gronwall inequality, for any $t\in[0,T]$,
	\begin{align*}
		\vert \varphi(t)\vert ^2\leq K(x_0,T,L)\left(1+S_T(\varphi)\right)e^{K(x_0,T,L)\left(1+S_T(\varphi)\right)t}
		\leq K(x_0,T,L)e^{K(x_0,T,L)S_T(\varphi)},
	\end{align*}
	where we have used the fact $1+x\leq e^x$ for any $x\in\mbb R$. Further, we obtain $\|\varphi\|_0\leq K(x_0,T,L)e^{K(x_0,T,L)S_T(\varphi)}$. Hence, $\|b(\varphi)\|_0+\|\sigma(\varphi)\|_0\leq L(1+\|\varphi\|_0)\leq K(x_0,T,L)e^{K(x_0,T,L)S_T(\varphi)}$. Noting that  $\varphi'=b(\varphi)+\sigma(\varphi)f$ and $\sqrt{2S_T(\varphi)}\leq1+S_T(\varphi)\leq e^{S_T(\varphi)}$, we have
	\begin{align*}
		\|\varphi'\|_{\mbf L^2}
		\leq \sqrt{T}\|b(\varphi)\|_0+\|\sigma(\varphi)\|_0\sqrt{2S_T(\varphi)}
		\leq K(x_0,T,L)e^{K(x_0,T,L)S_T(\varphi)}.
	\end{align*} 
	Thus $\|\varphi\|_{\mbf H^1}=\|\varphi'\|_{\mbf L^2}+\|\varphi\|_{\mbf L^2}\le \|\varphi'\|_{\mbf L^2}+\sqrt{T}\|\varphi\|_{0}\le K(x_0,T,L)e^{K(x_0,T,L)S_T(\varphi)}$, which completes the proof.
\end{proof}

\begin{pro}\label{STlower}
	For any sequence $\{\varphi_n\}_{n\in\mbb N^+}$  converging to some $\varphi$ with respect to (w.r.t.) the weak topology of $\mbf H^1(0,T;\mbb R^d)$, it holds that
	$\liminf\limits_{n\to\infty}S_T(\varphi_n)\ge S_T(\varphi).$
\end{pro}
\begin{proof}
	Assume that $\{\varphi_n\}_{n\in\mbb N^+}$  converges to some $\varphi$ weakly in $\mbf H^1(0,T;\mbb R^d)$, which means that $\lim\limits_{n\to\infty}\varphi_n=\varphi$ and $\lim\limits_{n\to\infty}\varphi_n'=\varphi'$ w.r.t. the weak topology of $\mbf L^2(0,T;\mbb R^d)$. Thus,  $\sup\limits_{n\in\mbb N^+}\|\varphi'_n\|_{\mbf L^2}<+\infty$.
	Since $\mbf H^1(0,T;\mbb R^d)$ is compactly embedded into $\mbf C([0,T];\mbb R^d)$, we have $\lim\limits_{n\to+\infty}\|\varphi_n-\varphi\|_0=0$ and thus $\sup\limits_{n\in\mbb N^+}\|\varphi_n\|_0<+\infty$.	This combined with \eqref{lip} and Proposition \ref{Prop}(1)  yields
	\begin{gather}
		\lim\limits_{n\to+\infty}\|b(\varphi_n)-b(\varphi)\|_0=0,\qquad
		\lim_{n\to\infty}\|\sigma^{-1}(\varphi_n)-\sigma^{-1}(\varphi)\|_0=0,\label{k1}\\
		\lim_{n\to\infty}\|\sigma^{-1}(\varphi_n)b(\varphi_n)-\sigma^{-1}(\varphi)b(\varphi)\|_0=0. \label{k2}
	\end{gather}
	It follows from \eqref{k1} and $\sup\limits_{n\in\mbb N^+}\|\varphi'_n\|_{\mbf L^2}<+\infty$ that  for any $f\in\mbf L^2(0,T;\mbb R^d)$,
	\begin{align*}
		&\;\lim_{n\to\infty}\big(\LL\sigma^{-1}(\varphi_n)\varphi_n',f\RR_{\mbf L^2}-\LL\sigma^{-1}(\varphi)\varphi,f\RR_{\mbf L^2}\big)\\
		=&\;\lim_{n\to\infty}\LL(\sigma^{-1}(\varphi_n)-\sigma^{-1}(\varphi))\varphi_n',f\RR_{\mbf L^2}+\lim_{n\to\infty}\LL\varphi_n'-\varphi',(\sigma^{-1}(\varphi))^\top f\RR_{\mbf L^2}=0,
	\end{align*}
	which implies
	\begin{align}\label{k3}
		\lim_{n\to\infty}\sigma^{-1}(\varphi_n)\varphi'_n=\sigma^{-1}(\varphi)\varphi',\quad~\text{weakly in}~\mbf L^2(0,T;\mbb R^d).
	\end{align}
	Using \eqref{k2} and \eqref{k3} yields
	\begin{align}\label{k4}
		\lim_{n\to\infty}\LL\sigma^{-1}(\varphi_n)\varphi'_n,\sigma^{-1}(\varphi_n)b(\varphi_n)\RR_{\mbf L^2}=\LL\sigma^{-1}(\varphi)\varphi',\sigma^{-1}(\varphi)b(\varphi)\RR_{\mbf L^2}.
	\end{align}
	Finally, combining \eqref{k2}-\eqref{k4} and the weakly lower semicontinuity of the norm, we arrive at
	\begin{align*}
		\liminf_{n\to\infty}S_T(\varphi_n)
		&=\frac{1}{2}\liminf_{n\to\infty}\|\sigma^{-1}(\varphi_n)\varphi_n'\|_{\mbf L^2}^2-\lim_{n\to\infty}\LL\sigma^{-1}(\varphi_n)\varphi'_n,\sigma^{-1}(\varphi_n)b(\varphi_n)\RR_{\mbf L^2}\\
		&\quad~+\frac{1}{2}\lim_{n\to\infty}\|\sigma^{-1}(\varphi_n)b(\varphi_n)\|_{\mbf L^2}^2\\
		&\ge \frac{1}{2}\|\sigma^{-1}(\varphi)\varphi'\|_{\mbf L^2}^2-\LL\sigma^{-1}(\varphi)\varphi',\sigma^{-1}(\varphi)b(\varphi)\RR_{\mbf L^2}+\frac{1}{2}\|\sigma^{-1}(\varphi)b(\varphi)\|_{\mbf L^2}^2\\
		&=S_T(\varphi).
	\end{align*}
	Thus the proof is complete.
\end{proof}

Equipped with Propositions \ref{STcoercive}-\ref{STlower}, we can use the classical variational theory (see e.g., \cite{variation}) to prove the existence of a minimizer of $S_T$.
\begin{lem}\label{STmin}
	There exists a function $\varphi^*\in\mbf H^1_{x_0,x}(0,T;\mbf R^d)$ such that 
	$$S_T(\varphi^*)=\inf_{\varphi\in\mbf H^1_{x_0,x}(0,T;\mbb R^d)}S_T(\varphi).$$
\end{lem}
\begin{proof}
	Denote $A:=\inf\limits_{\varphi\in\mbf H^1_{x_0,x}(0,T;\mbb R^d)}S_T(\varphi)$. It is easy to see that $A\le S_T(\Phi_L)<+\infty$, where $\Phi_L=x_0+\frac{t}{T}(x-x_0)$, $t\in[0,T]$. Then we can take a minimization  sequence $\{\varphi_n\}_{n\in\mbb N^+}\subseteq\mbf H^1_{x_0,x}(0,T;\mbb R^d)$ such that $\lim\limits_{n\to\infty}S_T(\varphi_n)=A$, and thus $\sup\limits_{n\in\mbb N^+}S_T(\varphi_n)\le K(A)$. Applying Proposition \ref{STcoercive}, one has
	$\sup\limits_{n\in\mbb N^+}\|\varphi_n\|_{\mbf H^1}\le K(A,C_0).$
	Consequently, there exists a subsequence $\{\varphi_{n_k}\}_{k\in\mbb N^+}$ of $\{\varphi_{n}\}_{n\in\mbb N^+}$ converging weakly to some $\varphi^*\in\mbf H^1(0,T;\mbb R^d)$, due to the reflexivity of   $\mbf H^1(0,T;\mbb R^d)$.
	Since $\mbf H^1(0,T;\mbb R^d)$ is compactly embedded into $\mbf C([0,T],\mbb R^d)$, $\lim\limits_{k\to\infty}\varphi_{n_k}=\varphi^*$ in $\|\cdot\|_0$-norm. This indicates that $\varphi(0)=x_0$ and $\varphi(T)=x$, and thus $\varphi^*\in\mbf H^1_{x_0,x}(0,T;\mbb R^d)$. Further, combining Proposition \ref{STlower}, we obtain
	$A\le S_T(\varphi^*)\le \liminf\limits_{k\to\infty}S_T(\varphi_{n_k})=A,$
	which finishes the proof.
\end{proof}

\subsection{Existence of minimizers of $S_{T,h}$ and $\hat{S}_{T,h}$}
In this part, we present the existence of minimizers for Problems III and IV. Before that, we  show that Problem III is equivalent to Problem IV.

\begin{lem}\label{equivalent}
	We have 
	\begin{align}\label{infequi}
		&\;\inf_{(\psi_1,\psi_2,\ldots,\psi_{N-1})\in\mbb R^{N-1}}S_{T,h}(\psi_1,\psi_2,\ldots,\psi_{N-1})\nonumber\\
		=&\;\inf_{\varphi\in\mbf H^1_{x_0,x}(0,T;\mbb R^d)}\hat{S}_{T,h}(\varphi)=\inf_{\varphi\in\mcal M_{x_0,x}(0,T;\mbb R^d)}\hat{S}_{T,h}(\varphi),
	\end{align}
	where $\mcal M_{x_0,x}(0,T;\mbb R^d):=\big\{u\in\mbf C([0,T];\mbb R^d):u(t)=u(t_n)+\frac{u(t_{n+1})-u(t_n)}{h}(t-t_n),~t\in[t_n,t_{n+1}],~n=0,1,\ldots,N-1,~\text{and}~u(0)=x_0,~u(T)=x\big\}$.
\end{lem}
\begin{proof}
	The H\"older inequality and \eqref{hatSTh} yield that	for any $\varphi\in\mbf H^1_{x_0,x}(0,T;\mbb R^d)$,
	\begin{align*}
		\hat{S}_{T,h}(\varphi)&=\frac{1}{2}\sum_{n=0}^{N-1}\int_{t_n}^{t_{n+1}}\big|\sigma^{-1}(\varphi(t_n))\big(\varphi'(t)-b((1-\theta)\varphi(t_n)+\theta\varphi(t_{n+1}))\big)\big|^2\ud t\\
		&\ge\frac{1}{2h}\sum_{n=0}^{N-1}\Big|\int_{t_n}^{t_{n+1}}\sigma^{-1}(\varphi(t_n))\big(\varphi'(t)-b((1-\theta)\varphi(t_n)+\theta\varphi(t_{n+1}))\big)\ud t\Big|^2\\
		&=\frac{1}{2h}\sum_{n=0}^{N-1}\Big|\sigma^{-1}(\varphi(t_n))\big(\varphi(t_{n+1})-\varphi(t_n)-hb((1-\theta)\varphi(t_n)+\theta\varphi(t_{n+1}))\big)\Big|^2\\
		&=S_{T,h}(\varphi(t_1),\varphi(t_2),\ldots,\varphi(t_{N-1})).
	\end{align*}
	Thus, 
	$\inf\limits_{(\psi_1,\psi_2,\ldots,\psi_{N-1})\in\mbb R^{N-1}}S_{T,h}(\psi_1,\psi_2,\ldots,\psi_{N-1})\le \inf\limits_{\varphi\in\mbf H^1_{x_0,x}(0,T;\mbb R^d)}\hat{S}_{T,h}(\varphi).$
	On the other hand, for any $(\psi_1,\psi_2,\ldots,\psi_{N-1})\in\mbb R^{N-1}$, define $\bar{\varphi}\in\mbf H^1_{x_0,x}(0,T;\mbb R^d)$ by
	\begin{align*}
		\bar{\varphi}(t)=\psi_{n}+\frac{\psi_{n+1}-\psi_{n}}{h}(t-t_n),\quad~t\in[t_n,t_{n+1}],~n=0,1,\ldots,N-1,
	\end{align*} 
	where $\psi_0=x_0$ and $\psi_N=x$. Further, it follows from \eqref{STh} that 
	\begin{align}\label{k5}
		&\;S_{T,h}(\psi_1,\psi_2,\ldots,\psi_{N-1})\nonumber\\
		=&\;\frac{1}{2}\sum_{n=0}^{N-1}\int_{t_n}^{t_{n+1}}\Big\vert \sigma^{-1}(\psi_n)\Big[\frac{\psi_{n+1}-\psi_n}{h}-b((1-\theta)\psi_n+\theta\psi_{n+1})\Big]\Big\vert ^2\ud t\nonumber\\
		=&\; \frac{1}{2}\sum_{n=0}^{N-1}\int_{t_n}^{t_{n+1}}\Big|\sigma^{-1}(\bar{\varphi}(\hat{t}))\big(\bar{\varphi}'(t)-b((1-\theta)\bar{\varphi}(\hat{t})+\theta\bar{\varphi}(\check{t}))\big)\Big|^2\ud t\nonumber\\
		=&\;\hat{S}_{T,h}(\bar{\varphi}),
	\end{align}
	which leads to
	$\inf\limits_{(\psi_1,\psi_2,\ldots,\psi_{N-1})\in\mbb R^{N-1}}S_{T,h}(\psi_1,\psi_2,\ldots,\psi_{N-1})\ge \inf\limits_{\varphi\in\mbf H^1_{x_0,x}(0,T;\mbb R^d)}\hat{S}_{T,h}(\varphi).$
	Thus, the first equality in \eqref{infequi} holds.
	Note that $\mbb R^{N-1}$ is isomorphic to $\mcal M_{x_0,x}(0,T;\mbb R^d)$, which together with
	\eqref{k5} yields the second equality in \eqref{infequi}. This finishes the proof.
\end{proof}

\begin{rem}\label{rem2.6}
	One can conclude from \eqref{infequi} and \eqref{k5} that there is a one-to-one correspondence  between 
	the set of minimizers of $S_{T,h}$ and that of $\hat S_{T,h}$  through the  relation
	$\varphi^*(t)=\psi^*_{n}+\frac{\psi^*_{n+1}-\psi^*_{n}}{h}(t-t_n),~t\in[t_n,t_{n+1}],~n=0,1,\ldots,N-1,$
	where $\psi^*_0=x_0$ and $\psi^*_N=x$. 
\end{rem}

The following two lemmas give the equi-coerciveness and weakly lower semicontinuity of $\hat{S}_{T,h}$.

\begin{lem}\label{SThcoer}  There exists some constant $C_1>0$ independent of $h$ such that for any $\varphi\in\mbf H^1_{x_0}(0,T;\mbb R^d)$ and $h\in(0,\frac{1}{2L}]$,
	$\|\varphi\|_{\mbf H^1}\leq C_1e^{C_1\hat{S}_{T,h}(\varphi)}.$
\end{lem}
\begin{proof}
	We use $K(x_0,T,L)$ to denote some constant depending on $x_0,T$ and $L$, but independent of the step-size $h$, which may vary from one place to somewhere. 
	
	Denote $g(t)=\sigma^{-1}(\varphi(\hat t))\left(\varphi'(t)-b\left((1-\theta)\varphi(\hat t)+\theta\varphi(\check t)\right)\right)$ for a.e. $t\in[0,T]$. Then $\|g\|_{\mbf L^2}^2=2\hat{S}_{T,h}(\varphi)$ and for any $t\in[0,T]$,
	\begin{align}\label{sec3eq1}
		\varphi(t)=x_0+\int_0^tb\left((1-\theta)\varphi(\hat s)+\theta\varphi(\check s)\right)\ud s+\int_0^t\sigma(\varphi(\hat s))g(s)\ud s.
	\end{align}
	Hence for any $n=1,2,\ldots,N$,
	\begin{align*}
		\varphi(t_n)=\varphi(t_{n-1})+hb\left((1-\theta)\varphi(t_{n-1})+\theta\varphi(t_n)\right)+\sigma(\varphi(t_{n-1}))\int_{t_{n-1}}^{t_n}g(t)\ud t.
	\end{align*}
	It follows from \eqref{linear} that 
	\begin{align*}
		\vert \varphi(t_n)\vert &\leq \vert \varphi(t_{n-1})\vert +hL\left(1+\vert \varphi(t_{n-1})\vert +\vert \varphi(t_{n})\vert \right)+L(1+\vert \varphi(t_{n-1})\vert )\int_{t_{n-1}}^{t_n}\vert g(t)\vert \ud t\\
		&=\left(1+hL+L\int_{t_{n-1}}^{t_n}\vert g(t)\vert \ud t\right)\vert \varphi(t_{n-1})\vert +Lh\vert \varphi(t_{n})\vert +Lh+L\int_{t_{n-1}}^{t_n}\vert g(t)\vert \ud t.
	\end{align*}
	Thus for any $h\in(0,\frac{1}{2L}]$ and $n=1,2,\ldots,N$,
	\begin{align*}
		&\phantom{=}\vert \varphi(t_n)\vert \\
		&\leq \frac{\left(1+hL+L\int_{t_{n-1}}^{t_n}\vert g(t)\vert \ud t\right)\vert \varphi(t_{n-1})\vert }{1-Lh}+\frac{Lh}{1-Lh}+\frac{L}{1-Lh}\int_{t_{n-1}}^{t_n}\vert g(t)\vert \ud t\\
		&=\vert \varphi(t_{n-1})\vert +\frac{2Lh+L\int_{t_{n-1}}^{t_n}\vert g(t)\vert \ud t}{1-Lh}\vert \varphi(t_{n-1})\vert +\frac{Lh}{1-Lh}+\frac{L}{1-Lh}\int_{t_{n-1}}^{t_n}\vert g(t)\vert \ud t\\
		&\leq\vert \varphi(t_{n-1})\vert +\left(4Lh+2L\int_{t_{n-1}}^{t_n}\vert g(t)\vert \ud t\right)\vert \varphi(t_{n-1})\vert +2Lh+2L\int_{t_{n-1}}^{t_n}\vert g(t)\vert \ud t.
	\end{align*}
	Setting $k_n=4Lh+2L\int_{t_{n}}^{t_{n+1}}\vert g(t)\vert \ud t$, $n=0,1,\ldots,N-1$ and by iteration, we  have
	\begin{align*}
		\vert \varphi(t_n)\vert 
		&\leq\vert x_0\vert +\sum_{j=0}^{n-1}k_j\vert \varphi(t_j)\vert +\sum_{j=0}^{n-1}k_j\quad\forall~n=1,2,\ldots,N.
	\end{align*}
	It follows from \cite[Lemma 1.4.2]{AA94} that 
	$\sup\limits_{n=0,1,\ldots,N}\vert \varphi(t_n)\vert \leq \Big(\vert x_0\vert +\sum\limits_{j=0}^{n-1}k_j\Big)e^{\sum_{j=0}^{n-1}k_j}.$
	Since
	\begin{align*}
		\sum_{j=0}^{N-1}k_j&=4LT+2L\int_0^T\vert g(t)\vert \ud t\leq 4LT+L^2+T\|g\|_{\mbf L^2}^2\\
		&=4LT+L^2+2T\hat{S}_{T,h}(\varphi)\leq K(T,L)(1+\hat{S}_{T,h}(\varphi)),
	\end{align*}
	we obtain
	\begin{align*}
		\sup\limits_{n=0,1,\ldots,N}\vert \varphi(t_n)\vert &\leq\left(\vert x_0\vert +K(T,L)(1+\hat{S}_{T,h}(\varphi))\right)e^{K(T,L)(1+\hat{S}_{T,h}(\varphi))}\\
		&\leq K(x_0,T,L)e^{K(x_0,T,L)\hat{S}_{T,h}(\varphi)}.
	\end{align*}
	As a consequence,
	\begin{align*}
		&\phantom{\leq}\sup\limits_{t\in[0,T]}\left\vert b\left((1-\theta)\varphi(\hat t)+\theta\varphi(\check t)\right)\right\vert+\sup\limits_{t\in[0,T]}\left\vert \sigma(\varphi(\hat t))\right\vert\\
		& \leq L\Big(1+\sup\limits_{n=0,1,\ldots,N}\vert \varphi(t_n)\vert \Big)
		\leq K(x_0,T,L)e^{K(x_0,T,L)\hat{S}_{T,h}(\varphi)}.
	\end{align*}
	Combining the above formulas, \eqref{sec3eq1}, the H\"older inequality and the fact $\|g\|_{\mbf L^2}=\sqrt{2\hat{S}_{T,h}(\varphi)}\leq e^{\hat{S}_{T,h}(\varphi)}$, we get
	$\|\varphi\|_{\mbf L^2}\le \sqrt{T}\|\varphi\|_0\le  K(x_0,T,L)e^{K(x_0,T,L)\hat{S}_{T,h}(\varphi)}.$
	In addition,	by \eqref{sec3eq1}, $\varphi'(t)=b\left((1-\theta)\varphi(\hat t)+\theta\varphi(\check t)\right)+\sigma(\varphi(\hat t))g(t)$ for a.e. $t\in[0,T]$. Combining the previous estimates gives
	\begin{align*}
		\|\varphi'\|_{\mbf L^2}&\leq K(x_0,T,L)e^{K(x_0,T,L)\hat{S}_{T,h}(\varphi)}\sqrt{T}+K(x_0,T,L)e^{K(x_0,T,L)\hat{S}_{T,h}(\varphi)}\|g\|_{\mbf L^2}\\
		&\leq K(x_0,T,L)e^{K(x_0,T,L)\hat{S}_{T,h}(\varphi)}.
	\end{align*}
	Thus, we obtain the desired conclusion.
\end{proof}

\begin{lem}\label{SThlower}
	For any sequence $\{\varphi_n\}_{n\in\mbb N^+}$  converging weakly to some $\varphi$ in $\mbf H^1(0,T;\mbb R^d)$, it holds that
	$\liminf\limits_{n\to\infty}\hat{S}_{T,h}(\varphi_n)\ge \hat{S}_{T,h}(\varphi).$
\end{lem}
\begin{proof}
	Assume that $\{\varphi_n\}_{n\in\mbb N^+}$  converges to some $\varphi$ weakly in $\mbf H^1(0,T;\mbb R^d)$. Then $\{\varphi_n\}_{n\in\mbb N^+}$ converges to $\varphi$ in $\|\cdot\|_0$-norm due to the Sobolev compact embedding. As a result,
	$\lim\limits_{n\to+\infty}\sup\limits_{t\in[0,T]}|b((1-\theta)\varphi_n(\hat{t})+\theta\varphi_n(\check{t}))-b((1-\theta)\varphi(\hat{t})+\theta\varphi(\check{t}))|=0$ and
	$\lim\limits_{n\to\infty}\sup\limits_{t\in[0,T]}|\sigma^{-1}(\varphi_n(\hat{t}))-\sigma^{-1}(\varphi(\hat{t}))|=0.$ 
	The remainder of proof resembles that of Proposition \ref{STlower}, and thus is omitted.
\end{proof}

With previous preparations, one can use the same arguments as in the proof of Lemma \ref{STmin} to show that $\hat{S}_{T,h}$ admits at least a minimizer.

\begin{lem}\label{SThmin}
	For any  $h\in(0,\frac{1}{2L}]$, there exists $\varphi^*_h\in\mbf H^1_{x_0,x}(0,T;\mbb R^d)$ such that
	$$\hat{S}_{T,h}(\varphi^*_h)=\inf_{\varphi\in\mbf H^1_{x_0,x}(0,T;\mbb R^d)}\hat{S}_{T,h}(\varphi).$$
\end{lem}

Further, we show that $S_{T,h}$ admits a minimizer.
\begin{lem}\label{SThmin2}
	For any given $h\in(0,\frac{1}{2L}]$, there exists $(\psi_1^*,\psi_2^*,\ldots,\psi_{N-1}^*)\in\mbb R^{N-1}$ such that
	$S_{T,h}(\psi_1^*,\psi_2^*,\ldots,\psi_{N_1}^*)=\inf\limits_{(\psi_1,\psi_2,\ldots,\psi_{N-1})\in\mbb R^{N-1}}S_{T,h}(\psi_1,\psi_2,\ldots,\psi_{N-1}).$
\end{lem}
\begin{proof}
	By Lemma \ref{SThmin}, $\hat S_{T,h}$ admits a minimizer $\varphi_h^*\in\mbf H^1_{x_0,x}(0,T;\mbb R^d)$. 
	Define
	$\widetilde{\varphi}_h\in\mcal M_{x_0,x}(0,T;\mbb R^d)$ as
	the linear interpolation of $\varphi^*_h$, i.e., 
	$\widetilde{\varphi}_h(t)=\varphi^*_h(t_n)+\frac{\varphi^*_h(t_{n+1})-\varphi^*_h(t_n)}{h}(t-t_n),~t\in[t_n,t_{n+1}],~n=0,1,\ldots,N-1$. Then it holds that
	\begin{align*}
		&\phantom{=}\hat{S}_{T,h}(\widetilde{\varphi}_h)\\
		&=\frac{1}{2}\sum_{n=0}^{N-1}\int_{t_n}^{t_{n+1}}\Big\vert \sigma^{-1}({\varphi}^*_h(t_n))\Big[\frac{1}{h}(\varphi^*_h(t_{n+1})-\varphi^*_h(t_n))-b((1-\theta){\varphi}^*_h(t_n)+\theta{\varphi}^*_h(t_{n+1}))\Big]\Big\vert ^2\ud t\\
		&=\frac{1}{2h}\sum_{n=0}^{N-1}\Big\vert \int_{t_n}^{t_{n+1}}\sigma^{-1}({\varphi}^*_h(t_n))\Big[(\varphi^*_h)'(t)-b((1-\theta){\varphi}^*_h(t_n)+\theta{\varphi}^*_h(t_{n+1}))\Big]\ud t\Big\vert ^2\\
		&\le \frac{1}{2}\sum_{n=0}^{N-1}\int_{t_n}^{t_{n+1}}\Big\vert \sigma^{-1}({\varphi}^*_h(t_n))\Big[(\varphi^*_h)'(t)-b((1-\theta){\varphi}^*_h(t_n)+\theta{\varphi}^*_h(t_{n+1}))\Big]\Big\vert ^2\ud t\\
		&=\hat{S}_{T,h}(\varphi^*_h).
	\end{align*}
	This implies that $\widetilde{\varphi}_h\in \mcal M_{x_0,x}(0,T;\mbb R^d)$ is also a minimizer of $\hat{S}_{T,h}$. Then by Remark \ref{rem2.6}, $S_{T,h}$ admits a minimizer, and the proof is complete.	
\end{proof}

\section{Convergence analysis}\label{Sec3}
In this section, we are devoted to analyzing the convergence of the minimum  and minimizers of $\hat{S}_{T,h}$. 
Hereafter, let $K(R)$ denote a generic constant depending on the parameter $R$ but independent of the step-size $h$, which may vary from one place to another. Denote  $\mathbf{B}_R:=\left\{\varphi\in\mbf H^1_{x_0}(0,T;\mbb R^d):\|\varphi\|_{\mbf H^1}\leq R\right\}$ and we have the following estimates, which establish the locally  uniform convergence of $\hat{S}_{T,h}$ to $S_T$.
\begin{lem}\label{localerror}
	For any $R>0$, there exists some constant $K_1(R)>0$ such that for any $h\in(0,1]$,
	\begin{align}\label{error1}
		\sup_{\varphi\in \mathbf{B}_R}\left\vert S_T(\varphi)-\hat{S}_{T,h}(\varphi)\right\vert \leq K_1(R)h^{1/2}.
	\end{align} 
	In particular, if $\sigma$ is an  invertible constant matrix, we have that for any $R>0$,  there is some constant $K_2(R)>0$ such that for any $h\in(0,1]$,
	\begin{align}\label{error2}
		\sup_{\varphi\in \mathbf{B}_R}\left\vert S_T(\varphi)-\hat{S}_{T,h}(\varphi)\right\vert \leq K_2(R)h.
	\end{align} 
\end{lem}
\begin{proof}
	Denote $f(t):=\sigma^{-1}(\varphi(t))\left(\varphi'(t)-b(\varphi(t))\right)$ and $g(t):=\sigma^{-1}(\varphi(\hat t))\big(\varphi'(t)\\-b((1-\theta)\varphi(\hat t)+\theta \varphi(\check t))\big)$ for a.e. $t\in[0,T]$. Then $S_T(\varphi)=\frac{1}{2}\|f\|_{\mbf L^2}^2$ and $\hat{S}_{T,h}(\varphi)=\frac{1}{2}\|g\|_{\mbf L^2}^2$. Hence,
	$\left\vert S_T(\varphi)-\hat{S}_{T,h}(\varphi)\right\vert 
	\leq \frac{1}{2}\left(\|f\|_{\mbf L^2}+\|g\|_{\mbf L^2}\right)\|f-g\|_{\mbf L^2}.$
	By means of Proposition \ref{Prop}(2), for any $\varphi\in \mathbf{B}_R$, $\|\varphi\|_0\leq \vert x_0\vert +\sqrt{T}\|\varphi\|_{\mbf H^1}\leq \vert x_0\vert +\sqrt{T}R$. Since $\sigma^{-1}$ is continuous, there is some constant $K(R)>0$ such that
	\begin{align}\label{sec4eq1}
		\|\sigma^{-1}\circ\varphi\|_0\leq\sup_{|x|\le \vert x_0\vert +\sqrt{T}R}\vert \sigma^{-1}(x)\vert \leq K(R)\quad\forall ~\varphi\in \mathbf{B}_R.
	\end{align}	
	Therefore, 
	\begin{align*}
		\|f\|_{\mbf L^2}\leq& K(R)\left(\|\varphi'\|_{\mbf L^2}+\|b(\varphi)\|_{\mbf L^2}\right)
		\leq K(R)\left(\|\varphi\|_{\mbf H^1}+\sqrt{T}\|b(\varphi)\|_0\right)\\
		\leq& K(R)\left(\|\varphi\|_{\mbf H^1}+\sqrt{T}L(1+\|\varphi\|_0)\right)
		\leq K(R)\quad\forall~\varphi\in \mathbf{B}_R.
	\end{align*}
	Noting that $\sup\limits_{t\in[0,T]}\left\vert \sigma^{-1}(\varphi(\hat t))\right\vert \leq \|\sigma^{-1}(\varphi)\|_0\leq K(R)$ for every $\varphi\in \mathbf{B}_R$, we have
	\begin{align*}
		\|g\|_{\mbf L^2}\leq& K(R)\left(\|\varphi\|_{\mbf H^1}+\sqrt{T}\sup_{t\in[0,T]}\left\vert b\left((1-\theta)\varphi(\hat t)+\theta\varphi(\check t)\right)\right\vert \right)\\
		\leq & K(R)\left(\|\varphi\|_{\mbf H^1}+L\sqrt{T}\left(1+\|\varphi\|_0\right)\right)
		\leq K(R)\quad\forall~\varphi\in \mathbf{B}_R.
	\end{align*}
	Accordingly, it holds that
	\begin{align}\label{sec4eq2}
		\left\vert S_T(\varphi)-\hat{S}_{T,h}(\varphi)\right\vert \leq K(R)\|f-g\|_{\mbf L^2}\quad\forall~\varphi\in \mathbf{B}_R.
	\end{align}
	Next, we decompose $f-g$ into 
	\begin{align}\label{sec4eq3}
		f(t)-g(t)=&\left(\sigma^{-1}(\varphi(t))-\sigma^{-1}(\varphi(\hat t))\right)\left(\varphi'(t)-b(\varphi(t))\right)\nonumber\\
		&-\sigma^{-1}(\varphi(\hat t))\left(b(\varphi(t))-b\left((1-\theta)\varphi(\hat t)+\theta\varphi(\check{t})\right)\right).
	\end{align}
	By Proposition \ref{Prop}(2), for any $t\in[0,T]$ and $\varphi\in \mathbf{B}_R$,
	\begin{align*}
		\left\vert \varphi(t)-\varphi(\hat t)\right\vert \leq (t-\hat t)^{1/2}\left(\int_{\hat t}^{t}\vert \varphi'(s)\vert ^2\ud s\right)^{1/2}
		\leq\|\varphi\|_{\mbf H^1}h^{1/2}\leq Rh^{1/2}
	\end{align*}
	and
	\begin{align*}
		&\left\vert \varphi(t)-\left((1-\theta)\varphi(\hat t)+\theta\varphi(\check t)\right)\right\vert 
		\leq(1-\theta)\vert \varphi(t)-\varphi(\hat t)\vert +\theta\vert \varphi(t)-\varphi(\check t)\vert \\
		\leq&(1-\theta)(t-\hat t)^{1/2}\left(\int_{\hat t}^{t}\vert \varphi'(s)\vert ^2\ud s\right)^{1/2}+\theta(\check{t}-t)^{1/2}\left(\int_{ t}^{\check t}\vert \varphi'(s)\vert ^2\ud s\right)^{1/2}\\
		\leq &(1-\theta)\|\varphi\|_{\mbf H^1}h^{1/2}+\theta\|\varphi\|_{\mbf H^1}h^{1/2}
		\leq Rh^{1/2}.
	\end{align*}
	Notice that $\sigma^{-1}$ is locally Lipschitz continuous due to Proposition \ref{Prop}(1) and $\|\varphi\|_0\leq \vert x_0\vert +\sqrt{T}R$ provided that $\varphi\in \mathbf{B}_R$. There exists $L_R>0$ such that for any $\varphi\in \mathbf{B}_R$,
	\begin{align}\label{sec4eq4}
		\sup_{t\in[0,T]}\left\vert \sigma^{-1}(\varphi(t))-\sigma^{-1}(\varphi(\hat t))\right\vert \leq L_R\sup_{t\in[0,T]}\vert \varphi(t)-\varphi(\hat t)\vert \leq K(R)h^{1/2}.
	\end{align}
	Since $b$ is globally Lipschitz continuous,
	\begin{align}\label{sec4eq5}
		&\;\sup_{t\in[0,T]}\left\vert b(\varphi(t))-b\left((1-\theta)\varphi(\hat t)+\theta\varphi(\check{t})\right)\right\vert \nonumber\\
		\leq&\; L\sup_{t\in[0,T]}\left\vert \varphi(t)-\left((1-\theta)\varphi(\hat t)+\theta\varphi(\check t)\right)\right\vert \leq K(R)h^{1/2}.
	\end{align}
	Combining \eqref{sec4eq1} and \eqref{sec4eq3}-\eqref{sec4eq5} leads to
	\begin{align}\label{sec4eq6}
		\|f-g\|_{\mbf L^2}&\leq \sup_{t\in[0,T]}\left\vert \sigma^{-1}(\varphi(t))-\sigma^{-1}(\varphi(\hat t))\right\vert \left(\|\varphi'\|_{\mbf L^2}+\|b(\varphi)\|_{\mbf L^2}\right)\nonumber\\
		&\quad+\sqrt{T}\|\sigma^{-1}(\varphi)\|_0\sup_{t\in[0,T]}\left\vert b(\varphi(t))-b\left((1-\theta)\varphi(\hat t)+\theta\varphi(\check{t})\right)\right\vert \nonumber\\
		&\leq K(R)h^{1/2}\left(\|\varphi\|_{\mbf H^1}+\sqrt{T}L(1+\|\varphi\|_0)\right)+K(R)h^{1/2}\nonumber\\
		&\leq K(R)h^{1/2}\quad\forall~\varphi\in \mathbf{B}_R,
	\end{align}
	where we have used the inequality $\|\varphi\|_0\leq \vert x_0\vert +\sqrt{T}R$ provided that $\varphi\in \mathbf{B}_R$. Plugging \eqref{sec4eq6} into \eqref{sec4eq2} yields \eqref{error1}.
	
	For the case that $\sigma$ is an  invertible constant matrix, we still have
	\begin{align*}
		\vert S_T(\varphi)-\hat{S}_{T,h}(\varphi)\vert \leq K(R)\|f-g\|_{\mbf L^2}\quad\forall ~\varphi\in \mathbf{B}_R.
	\end{align*}
	In order to prove \eqref{error2}, it suffices to show that $\|f-g\|_{\mbf L^2}\leq K(R)h$. Notice that in this case, $f(t)-g(t)=\sigma^{-1}\big(b\left((1-\theta)\varphi(\hat t)+\theta\varphi(\check t)\right)-b(\varphi(t))\big)$ for a.e. $t\in[0,T]$. Hence, for any $t\in[0,T]$,
	\begin{align*}
		\vert f(t)-g(t)\vert &\leq L\left\vert (1-\theta)\varphi(\hat t)+\theta\varphi(\check t)-\varphi(t)\right\vert \leq L\left(\vert \varphi(t)-\varphi(\hat t)\vert +\vert \varphi(\check{t})-\varphi(t)\vert \right).
	\end{align*}
	Further, we obtain
	\begin{align*}
		\|f-g\|_{\mbf L^2}\leq L\left(\int_0^T\vert \varphi(t)-\varphi(\hat t)\vert ^2\ud t\right)^{1/2}+L\left(\int_0^T\vert \varphi(\check t)-\varphi(t)\vert ^2\ud t\right)^{1/2}.
	\end{align*}
	By Proposition \ref{Prop}(2),
	\begin{align}\label{w12}\notag
		&\quad\ \int_0^T\vert \varphi(t)-\varphi(\hat t)\vert ^2\ud t=\sum_{n=0}^{N-1}\int_{t_n}^{t_{n+1}}\vert \varphi(t)-\varphi(t_n)\vert ^2\ud t\\\notag
		&\leq\sum_{n=0}^{N-1}\int_{t_n}^{t_{n+1}}(t-t_n)\int_{t_n}^{t}\vert \varphi'(s)\vert ^2\ud s\ud t
		\leq h\sum_{n=0}^{N-1}\int_{t_n}^{t_{n+1}}\int_{t_n}^{t_{n+1}}\vert \varphi'(s)\vert ^2\ud s\ud t\\
		&=h^2\sum_{n=0}^{N-1}\int_{t_n}^{t_{n+1}}\vert \varphi'(s)\vert ^2\ud s
		=h^2\|\varphi\|_{\mbf H^1}^2.
	\end{align}
	This indicates that $\left(\int_0^T\vert \varphi(t)-\varphi(\hat t)\vert ^2\ud t\right)^{1/2}\leq \|\varphi\|_{\mbf H^1}h\leq Rh$, for any $\varphi\in \mathbf{B}_R$. Similarly, one has  $\left(\int_0^T\vert \varphi(\check t)-\varphi(t)\vert ^2\ud t\right)^{1/2}\leq Rh$ for any $\varphi\in \mathbf{B}_R.$ As a consequence, we obtain that $\|f-g\|_{\mbf L^2}\leq 2LRh$, which completes the proof.
\end{proof}

Next, we give the convergence analysis of the minimum and minimizers of $\hat{S}_{T,h}$.  Our idea is to use the existence of minimizers and the equi-coerciveness of $S_{T,h}$ to reduce the error between minimums of $S_T$ and $\hat{S}_{T,h}$  to that between  $S_T$ and $\hat{S}_{T,h}$ on  bounded sets.

\begin{theo}\label{convergence}
	We have the following.
	\begin{itemize}
		\item[(1)] There is some constant $C_2>0$ such that for any $h\in(0,\frac{1}{2L}\wedge1]$,
		\begin{align*}
			\bigg|\inf_{\varphi\in\mbf H^1_{x_0,x}(0,T;\mbb R^d)}S_T(\varphi)-\inf_{\varphi\in\mbf H^1_{x_0,x}(0,T;\mbb R^d)}\hat{S}_{T,h}(\varphi)\bigg|\leq C_2h^{1/2}.
		\end{align*}
		\item[(2)] In particular, if $\sigma$ is an  invertible constant matrix, then there is some constant $C_3>0$ such that for any $h\in(0,\frac{1}{2L}\wedge1]$,
		\begin{align*}
			\bigg|\inf_{\varphi\in\mbf H^1_{x_0,x}(0,T;\mbb R^d)}S_T(\varphi)-\inf_{\varphi\in\mbf H^1_{x_0,x}(0,T;\mbb R^d)}\hat{S}_{T,h}(\varphi)\bigg| \leq C_3h.
		\end{align*}
	\end{itemize}
\end{theo}

\begin{proof}
	(1) 
	By Lemma \ref{STmin}, there exists  $\varphi^*\in\mbf H^1_{x_0,x}(0,T;\mbb R^d)$ such that 
	\begin{align*}
		S_T(\varphi^*)=\inf_{\varphi\in\mbf H^1_{x_0,x}(0,T;\mbb R^d)}S_T(\varphi).
	\end{align*}
	It follows from Lemma \ref{localerror} that 
	\begin{align}\label{k6}
		|S_T(\varphi^*)-\hat{S}_{T,h}(\varphi^*)|\le K_1h^{1/2}\quad\forall~h\in(0,\frac{1}{2L}\wedge1],
	\end{align}
	for some constant $K_1>0$.	Thus, for any $h\in(0,\frac{1}{2L}\wedge1]$, 
	\begin{align}\label{sec4eq8}
		\inf_{\varphi\in\mbf H^1_{x_0,x}(0,T;\mbb R^d)}\hat{S}_{T,h}(\varphi)\leq \hat{S}_{T,h}(\varphi^*)&\leq S_T(\varphi^*)+K_1h^{1/2}\nonumber\\
		&=\inf_{\varphi\in\mbf H^1_{x_0,x}(0,T;\mbb R^d)}S_T(\varphi)+K_1h^{1/2}.
	\end{align}
	
	On the other hand, it follows from  Lemma \ref{SThmin} that   for any $h\in(0,\frac{1}{2L}\wedge1]$, there exists $\varphi^*_h\in\mbf H^1_{x_0,x}(0,T;\mbb R^d)$ such that
	$\hat S_{T,h}(\varphi^*_h)=\inf\limits_{\varphi\in\mbf H^1_{x_0,x}(0,T;\mbb R^d)}\hat S_{T,h}(\varphi).$
	This combined with \eqref{sec4eq8} gives
	\begin{align*}
		\hat S_{T,h}(\varphi^*_h)\le m_T+K_1h^{1/2}\le m_T+K_1\quad\forall~h\in(0,\frac{1}{2L}\wedge1],
	\end{align*}
	where $m_T:=\inf\limits_{\varphi\in\mbf H^1_{x_0,x}(0,T;\mbb R^d)}S_T(\varphi)$.
	Further, by Lemma \ref{SThcoer}, 
	\begin{align*}
		\|\varphi^*_h\|_{\mbf H^1}\leq C_1e^{C_1\hat{S}_{T,h}(\varphi^*_h)}\leq K(m_T)\quad\forall~ h\in(0,\frac{1}{2L}\wedge1].
	\end{align*}
	Thus, we  can use Lemma \ref{localerror} to get
	\begin{align}\label{sec4eq9}
		\left\vert S_T(\varphi^*_h)-\hat{S}_{T,h}(\varphi^*_h)\right\vert \leq K_2h^{1/2}\quad\forall~h\in(0,\frac{1}{2L}\wedge1],
	\end{align}
	for some $K_2>0$.
	Consequently, one has that for any $h\in(0,\frac{1}{2L}\wedge1]$, 
	\begin{align*}
		\inf\limits_{\varphi\in\mbf H^1_{x_0,x}(0,T;\mbb R^d)}S_T(\varphi)\leq S_T(\varphi^*_h)
		&\leq \hat{S}_{T,h}(\varphi^*_h)+K_2h^{1/2}\\
		&=\inf\limits_{\varphi\in\mbf H^1_{x_0,x}(0,T;\mbb R^d)}\hat S_{T,h}(\varphi)+K_2h^{1/2}.
	\end{align*}
	From the above formula and \eqref{sec4eq8}, it follows that for any $h\in(0,\frac{1}{2L}\wedge1]$,
	\begin{align}\label{sec4eq10}
		\bigg|\inf_{\varphi\in\mbf H^1_{x_0,x}(0,T;\mbb R^d)}S_T(\varphi)-\inf_{\varphi\in\mbf H^1_{x_0,x}(0,T;\mbb R^d)}\hat{S}_{T,h}(\varphi)\bigg|\leq (K_1+ K_2)h^{1/2}.
	\end{align}

	(2) In this case, one can use \eqref{error2} to improve the estimates of \eqref{k6} and $\eqref{sec4eq9}$. More precisely, one can similarly prove that there exists some $K_3>0$ such that 
	\begin{gather*}
		\big|S_T(\varphi^*)-\hat S_{T,h}(\varphi^*)\big|\le K_3h\quad\forall~h\in(0,\frac{1}{2L}\wedge1],\\
		\big|S_T(\varphi_h^*)-\hat S_{T,h}(\varphi_h^*)\big|\le K_3h\quad\forall~h\in(0,\frac{1}{2L}\wedge1].
	\end{gather*}
	Analogous to the proof of \eqref{sec4eq10}, we obtain the second conclusion, which
	completes the proof.
\end{proof}

We close this section by presenting the convergence of minimizers of $\hat S_{T,h}$ as $h\to0$.
\begin{theo}\label{convergence2}
	Let $\varphi_h\in\mbf H^1_{x_0,x}(0,T;\mbb R^d)$ be the minimizer of $\hat{S}_{T,h}$, $h\in(0,\frac{1}{2L}\wedge 1]$. Then there is a subsequence of $\{\varphi_h\}$ that converges to  some minimizer of $S_T$ w.r.t. the weak topology of $\mbf H^1(0,T;\mbb R^d)$. Moreover, if $\varphi^*$ is the unique minimizer of $S_T$, then $\{\varphi_h\}$ converges weakly in $\mbf H^1_{x_0,x}(0,T;\mbb R^d)$ to $\varphi^*$.	
\end{theo}
\begin{proof}
	It follows from Theorem \ref{convergence} that
	\begin{align}\label{k7}
		\lim_{h\to0}\hat{S}_{T,h}(\varphi_h)=\lim_{h\to0}\inf_{\varphi\in\mbf H^1_{x_0,x}(0,T;\mbb R^d)}\hat{S}_{T,h}(\varphi)=\inf_{\varphi\in\mbf H^1_{x_0,x}(0,T;\mbb R^d)}S_{T}(\varphi).
	\end{align}
	This implies that there exists some $h_0>0$ such that
	$\sup\limits_{h\in(0,h_0)}\hat{S}_{T,h}(\varphi_h)<+\infty$. Then an application of Lemma \ref{SThcoer} yields
	$\sup\limits_{h\in(0,h_0)}\|\varphi_h\|_{\mbf H^1}<+\infty$. Thus, there is a subsequence $\{\varphi_{h_n}\}_{n\in\mbb N^+}$ ($\lim\limits_{n\to\infty}h_n=0$) of $\{\varphi_h\}$ that converges weakly to some $\varphi_0\in\mbf H^1(0,T;\mbb R^d)$. Since $\mbf H^1(0,T;\mbb R^d)$ is compactly embedded into $\mbf C([0,T],\mbb R^d)$, $\lim\limits_{n\to\infty}\|\varphi_{h_n}-\varphi_0\|_0=0$, and thus $\varphi_0\in\mbf H^1_{x_0,x}(0,T;\mbb R^d)$. Further, by  $R:=\sup\limits_{n\in\mbb N^+}\|\varphi_{h_n}\|_{\mbf H^1}<+\infty$ and Lemma \ref{localerror}, we  have
	\begin{align}\label{k8}
		\lim_{n\to\infty}|S_T(\varphi_{h_n})-\hat{S}_{T,h_n}(\varphi_{h_n})|\le \lim_{n\to\infty}\sup_{\varphi\in \mbf B_R}|S_T(\varphi)-\hat{S}_{T,h_n}(\varphi)|=0.
	\end{align}
	It follows from Proposition \ref{STlower} and \eqref{k7}-\eqref{k8} that
	\begin{align*}
		S_T(\varphi_0)\le \liminf_{n\to\infty}S_T(\varphi_{h_n})&=\lim_{n\to\infty}(S_T(\varphi_{h_n})-\hat{S}_{T,h_n}(\varphi_{h_n}))+\liminf_{n\to\infty}\hat{S}_{T,h_n}(\varphi_{h_n})\\
		&=\inf_{\varphi\in\mbf H^1_{x_0,x}(0,T;\mbb R^d)}S_T(\varphi).
	\end{align*}
	Accordingly, $\varphi_0$ is a minimizer of $S_T$.
	
	In addition, if $\varphi^*$ is the unique minimizer of $S_T$, then every subsequence of $\{\varphi_h\}$ further contains a subsubsequence of $\{\varphi_h\}$ which converges weakly in $\mbf H^1_{x_0,x}(0,T;\mbb R^d)$ to $\varphi^*$, by using the conclusion of the first part. Thus, the whole minimizer sequence $\{\varphi_h\}$ converges weakly to $\varphi^*$. The proof is complete.
\end{proof}

We  note that the  convergence order of the minimum of $\hat{S}_{T,h}$  in the case of additive noises is $1$, higher than the  convergence order $1/2$ for the case of multiplicative noises.
As is shown in Theorem \ref{convergence}, in deriving the convergence order of  the minimum of $\hat{S}_{T,h}$, one prerequisite is the equi-coerciveness  of $\{\hat{S}_{T,h}\}_{h>0}$ in terms of $\mbf H^1$-norm (see Lemma \ref{SThcoer}). 
This allows us to reduce the convergence order of $\inf\limits_{\varphi\in\mbf H^1_{x_0,x}(0,T;\mbb R^d)}\hat S_{T,h}(\varphi)$ to the local uniform error order of $|\hat S_{T,h}-S_{T}|$ on $\mbf H^1(0,T;\mbb R^d)$.  	The key to getting $\sup\limits_{\varphi\in \mathbf{B}_R}|S_T(\varphi)-\hat{S}_{T,h}(\varphi)|=\mcal O(h)$, in the additive  noises case, lies in that $\left(\int_0^T\vert \varphi(t)-\varphi(\hat t)\vert ^2\ud t\right)^{1/2}\leq \|\varphi\|_{\mbf H^1}h$  for any $\varphi\in\mbf H^1_{x_0,x}(0,T;\mbb R^d)$ (see the proof of Lemma \ref{localerror}). This is not applicable to the multiplicative noises case, due to the presence  of   $\left(\sigma^{-1}(\varphi(\cdot))-\sigma^{-1}(\varphi(\hat \cdot))\right)\varphi'(\cdot)$. In order to improve the  estimate of
$\int_0^T\left\vert \left(\sigma^{-1}(\varphi(t))-\sigma^{-1}(\varphi(\hat t))\right)\varphi'(t)\right\vert ^2\ud t$ by the H\"older inequality, one needs $\varphi'$ to be $\mbf L^p$-integrable ($p\ge4$).  In fact, similar to the proof of \eqref{w12}, it holds that  for $p\ge4$,
\begin{align}\label{w1p}
	\sup_{\{\|\varphi\|_{\mbf W^{1,p}}\le R\}}\left\vert S_T(\varphi)-\hat{S}_{T,h}(\varphi)\right\vert \leq K(R)h 
\end{align}
for multiplicative noises case. Further, if the equi-coerciveness  of  $\{\hat{S}_{T,h}\}_{h>0}$ in terms of $\mbf W^{1,p}$-norm ($p\ge 4$) holds, i.e.,  there is $C>0$ such that for sufficiently small $h$ and any $\varphi\in\mbf W^{1,p}(0,T;\mbb R^d)$, 
\begin{equation}\label{eq:varphi}
	\|\varphi\|_{\mbf W^{1,p}}\le Ce^{C\hat{S}_{T,h}(\varphi)},
\end{equation}
then it is possible to obtain the first order convergence of $\inf\limits_{\varphi\in\mbf H^1_{x_0,x}(0,T;\mbb R^d)}\hat{S}_{T,h}(\varphi)$ for the case of multiplicative noises, as is done in the proof of Theorem \ref{convergence}.   
However, even for the simple case that $b\equiv0$ and $\sigma\equiv I_d$, one can only obtain  $\hat{S}_{T,h}(\varphi)=\frac{1}{2}\|\varphi^\prime\|_{\mbf L^2}^2$, which implies that \eqref{eq:varphi} fails to hold for $p\ge 4$. Hence the convergence order of $\inf\limits_{\varphi\in\mbf H^1_{x_0,x}(0,T;\mbb R^d)}\hat{S}_{T,h}$ is restricted to  $1/2$ for the moment. 
\begin{rem}
	We remark that  minimizers of $S_T$ will solve the Euler--Lagrange equation associated with $S_T$ (see \eqref{EulerS}). Thus, if the corresponding Euler--Lagrange equation admits a unique solution, then the minimizer of $S_T$ is unique.
\end{rem}

\section{Large deviation convergence of stochastic $\theta$-method}\label{Sec4}
In this section, we show that Theorem \ref{convergence} can be applied to analyzing the pointwise convergence of LDRFs of stochastic $\theta$-method for \eqref{SDE}. This reveals that the stochastic $\theta$-method can asymptotically preserve the large deviations principle (LDP) of $\{X^\epsilon(T)\}_{\epsilon>0}$.

We begin with a basic introduction to the LDP; see, e.g., \cite{ChenX,Dembo}. Let $\mcal X$ be a \emph{Polish space}, i.e., complete and separable metric space.  
A real-valued function $I:\mcal X\rightarrow[0,\infty]$ is called a \emph{rate function} if it is lower semicontinuous, i.e., for each $a\in[0,\infty)$, the level set $I^{-1}([0,a])$ is a closed subset of $\mcal X$. If all level sets $I^{-1}([0,a])$, $a\in[0,\infty)$, are compact, then $I$ is called a \emph{good rate function}.
Let $I$ be a rate function and $\{\mu_\epsilon\}_{\epsilon>0}$ a family of probability measures on  $\mcal X$. We say that $\{\mu_\epsilon\}_{\epsilon>0}$ satisfies an LDP on $\mcal X$ with the rate function $I$ if
\begin{flalign}
	(\rm{LDP 1})\qquad \qquad&\liminf_{\epsilon\to 0}\epsilon\ln(\mu_\epsilon(U))\geq-\inf I(U)\qquad\text{for every open}~ U\subseteq \mcal X,\nonumber&\\
	(\rm{LDP 2})\qquad\qquad &\limsup_{\epsilon\to 0}\epsilon\ln(\mu_\epsilon(C))\leq-\inf I(C)\qquad\text{for every closed}~ C\subseteq \mcal X.&\nonumber
\end{flalign}
Moreover,  a family of random variables  $\{Z_{\epsilon}\}_{\epsilon>0}$ valued on $\mcal X$ is said to satisfy an LDP with the rate function $I$, if its distribution $\{\mbf P\circ Z_{\epsilon}^{-1}\}_{\epsilon>0}$ satisfies  (LDP1) and  (LDP2). 

It is shown in \cite{LDPofSDE} that $\{X^\epsilon\}_{\epsilon>0}$ satisfies the LDP on $\mbf C_{x_0}([0,T];\mbb R^d)$ with the good rate function $J$ given by
\begin{align}\label{J}
	J(\varphi):=
	\begin{cases}
		S_T(\varphi),\quad\varphi\in\mbf H^1_{x_0}(0,T;\mbb R^d),\\
		+\infty,\qquad \varphi\in\mbf C_{x_0}([0,T],\mbb R^d)-\mbf H^1_{x_0}(0,T;\mbb R^d).
	\end{cases}
\end{align}
Define the coordinate map $\xi_T:\mbf C_{x_0}([0,T],\mbb R^d)\to \mbb R^d$ by $\xi_T(f)=f(T)$, for each $f\in\mbf C_{x_0}([0,T],\mbb R^d)$. Then we have $X^\epsilon(T)=\xi_T(X^\epsilon)$. Hence the continuity of the map $\xi_T$ and the contraction principle \cite[Theorem 4.2.1]{Dembo} give that 
$\left\{X^\epsilon(T)\right\}_{\epsilon>0}$ satisfies an LDP on $\mbb R^d$ with the good rate function $I(x)=\inf\limits_{\{\varphi\in\mbf C_{x_0}([0,T],\mbb R^d):\varphi(T)=x\}}J(\varphi),~x\in\mbb R^d.$
It can be verified that $I(x)$ is nothing but the minimum of $S_T$ on $\mbf H^1_{x_0,x}(0,T;\mbb R^d)$, i.e.,
\begin{align*}
	I(x)=\inf_{\varphi\in\mbf H^1_{x_0,x}(0,T;\mbb R^d)}S_T(\varphi)\quad\forall~x\in\mbb R^d.
\end{align*}
Let $X^0$ be the solution of the following  equation
\begin{align}\label{ODE}
	X^0(t)=x_0+\int_0^tb(X^0(s))\ud s,\quad t\in[0,T].
\end{align}
It can be verified that $X^\epsilon(T)$ converges to $X^0(T)$ in probability as $\epsilon\to 0$, i.e., for any $\delta>0$, $\lim\limits_{\epsilon\to0}\mbf P(\vert X^\epsilon(T)-X^0(T)\vert \geq\delta)=0$. As a direct consequence of the  LDP of  $\{X^\epsilon(T)\}_{\epsilon>0}$, one can characterize the decay speed of the probability $\mbf P(\vert X^\epsilon(T)-X^0(T)\vert \geq\delta)$ as $\epsilon\to0$ on an exponential scale.
\begin{cor}\label{cor2.6}
	The following properties hold.
	\begin{itemize}
		\item[(1)] $I(x)=0$ if and only if $x=X^0(T)$. 
		\item[(2)] Let $\delta>0$ be fixed and define $C(\delta):=\inf\limits_{\{x\in\mbb R^d:\vert x-X^0(T)\vert \geq\delta \}}I(x)$. Then $C(\delta)>0$ and for any $\eta\in(0,C(\delta))$, there exists some constant $\epsilon_0(\eta)>0$ such that for any $\epsilon\in(0,\epsilon_0(\eta))$,
		\begin{align}\label{decay}
			\mbf P(\vert X^\epsilon(T)-X^0(T)\vert \geq\delta)<e^{-\frac{1}{\epsilon}(C(\delta)-\eta)}.
		\end{align} 
	\end{itemize}
\end{cor}

The LDP of $\{X^\epsilon(T)\}_{\epsilon>0}$ means that for a Borel measurable set $A\subseteq\mbb R^d$, the hitting probability $\mbf P(X^\epsilon(T)\in A)\;\asymp\; e^{-\frac{1}{\epsilon}\inf\limits I(A)}$ $(\epsilon\rightarrow 0)$. 
A natural problem is whether a numerical approximation $Y^\epsilon_N$ of $X^{\epsilon}(T)$ can asymptotically preserve the exponential decay speed of $\mbf P(X^\epsilon(T)\in A)$, in the sense that  for any $N\in\mbb N^+$, $\{Y^\epsilon_N\}_{\epsilon>0}$ satisfies the LDP and its LDRF converges to $I$ as $N\to\infty$. Based on Theorem \ref{convergence}, we show in this section that the stochastic  $\theta$-method shares
the asymptotical preservation for the exponential decay of $\mbf P(X^\epsilon(T)\in A)$. This reveals  the practicality of using the stochastic $\theta$-method to simulate probabilities of rare events associated with \eqref{SDE}.

The stochastic $\theta$-method  for \eqref{SDE} reads
\begin{align}\label{theta}
	X^\epsilon_{n+1}=X^\epsilon_{n}+b\left((1-\theta)X^\epsilon_{n}+\theta X^\epsilon_{n+1}\right)h+\sqrt{\epsilon}\sigma(X^\epsilon_{n})\DW_n,~ n=0,1,\ldots,N-1,
\end{align}
where $\DW_n=W(t_{n+1})-W(t_n)$  is the increment of Brownian motion. 
Next we  give the  LDP of $\{X^\epsilon_N\}_{\epsilon>0}$.
\begin{theo}\label{LDPXN}
	For any $h\in(0,\frac{1}{2L}]$, $\{X^\epsilon_N\}_{\epsilon>0}$ satisfies the LDP on $\mbb R^d$ with the good rate function $I^h$ given by
	\begin{align}\label{Ih}
		I^h(x)=\inf_{\varphi\in\mbf H^1_{x_0,x}(0,T;\mbb R^d)}\hat{S}_{T,h}(\varphi)\quad\forall~x\in\mbb R^d.
	\end{align}
\end{theo}
\begin{proof}
	In this proof, we use $K(g_1,h,x_0)$ to denote a generic constant depending on $g_1,h$ and $x_0$ but independent of $g_2$, which may vary from one place to another, where $g_1$ and $g_2$ will be specified below. First we introduce the continuous version $\{\bar{X}^\epsilon(t),t\in[0,T]\}$ of the stochastic $\theta$-method \eqref{theta}:
	\begin{align*}
		\bar{X}^\epsilon(t)=x_0+\int_{0}^{t}b\left((1-\theta)\bar{X}^\epsilon(\hat s)+\theta\bar{X}^\epsilon(\check s)\right)\ud s+\sqrt{\epsilon}\int_{0}^{t}\sigma(\bar{X}^\epsilon(\hat s))\ud W(s)\quad\forall~t\in[0,T].
	\end{align*} 
	Recall that $\hat s:=\max\left(\left\{t_0,t_1,\ldots,t_N\right\}\cap[0,s]\right)$ and  $\check{s}:=\min\left(\left\{t_0,t_1,\ldots,t_N\right\}\cap[s,T]\right)$ for each $s\in[0,T]$.	Then it suffices to show that  $\{\bar{X}^\epsilon(T)\}_{\epsilon>0}$ satisfies the LDP with the good rate function $I^h$, due to $\bar{X}^\epsilon(T)=X^\epsilon_N$. 
	
	For any fixed $h\in(0,\frac{1}{2L}]$, define the map $F^h:\mbf C_{0}([0,T],\mbb R^m)\to\mbf C_{x_0}([0,T],\mbb R^d)$ by $f=F^h(g)$, where $f$ is the unique continuous solution of
	\begin{align*}
		f(t)=x_0+\int_0^tb\left((1-\theta)f(\hat s)+\theta f(\check s)\right)\ud s+\int_{0}^{t}\sigma(f(\hat s))\ud g(s)\quad\forall~t\in[0,T].
	\end{align*}
	Next we prove that $F^h$ is continuous. Let $g_1\in\mbf C_{0}([0,T],\mbb R^m)$ be fixed and denote $f_1=F^h(g_1)$.  By the definition of $F^h$, $f_1(0)=x_0$ and for any $t\in[t_n,t_{n+1}]$, $n=0,1,\ldots,N-1$,
	\begin{align}\label{f1}
		f_1(t)=f_1(t_n)+b\left((1-\theta)f_1(t_n)+\theta f_1(t_{n+1})\right)(t-t_n)+\sigma(f_1(t_n))(g_1(t)-g_1(t_n)).
	\end{align}
	It follows from  \eqref{linear} and \eqref{f1} that for $n=0,1,\ldots,N-1$,
	\begin{align*}
		\vert f_1(t_{n+1})\vert \leq\vert f_1(t_n)\vert +hL\left(1+\vert f_1(t_n)\vert +\vert f_1(t_{n+1})\vert \right)+2L\left(1+\vert f_1(t_n)\vert \right)\|g_1\|_0.
	\end{align*}
	Noting that $hL\leq\frac{1}{2}$ for $h\in(0,\frac{1}{2L}]$, we have that for $n=0,1,\ldots,N-1$,
	\begin{align*}
		\vert f_1(t_{n+1})\vert \leq\left(\frac{3}{2}+2L\|g_1\|_0\right)\vert f_1(t_n)\vert +\frac{1}{2}+2L\|g_1\|_0+\frac{1}{2}\vert f_1(t_{n+1})\vert,
	\end{align*}
	which yields that for $n=0,1,\ldots,N-1$,
	\begin{align*}
		\vert f_1(t_{n+1})\vert \leq\left(3+4L\|g_1\|_0\right)\vert f_1(t_n)\vert +1+4L\|g_1\|_0\leq C(g_1)\left(1+\vert f_1(t_n)\vert \right)
	\end{align*}
	with $C(g_1):=3+4L\|g_1\|_0$.
	By iteration, it holds that
	\begin{align*}
		\vert f_1(t_n)\vert 
		&\leq C(g_1)+C^2(g_1)+\cdots+ C^n(g_1)+C^n(g_1)\vert f_1(0)\vert ,\quad n=0,1,\ldots,N.
	\end{align*}
	Accordingly, one immediately has
	\begin{align*}
		\sup_{n=0,1,\ldots,N}\vert f_1(t_n)\vert \leq\sum_{i=1}^{N}C^i(g_1)+C^N(g_1)\vert x_0\vert =\frac{C^{T/h+1}(g_1)-C(g_1)}{C(g_1)-1}+C^{T/h}(g_1)\vert x_0\vert .
	\end{align*} 
	This is to say, $\sup\limits_{n=0,1,\ldots,N}\vert f_1(t_n)\vert \leq K(g_1,h,x_0)$, which along with \eqref{f1} gives that for any $t\in[t_n,t_{n+1}]$, $n=0,1,\ldots, N-1$,
	\begin{align*}
		\vert f_1(t)\vert &\leq\sup\limits_{n=0,1,\ldots,N}\vert f_1(t_n)\vert +hL\Big(1+(1-\theta)\sup\limits_{n=0,1,\ldots,N}\vert f_1(t_n)\vert +\theta\sup\limits_{n=0,1,\ldots,N}\vert f_1(t_n)\vert \Big)\\
		&\quad+2L\Big(1+\sup\limits_{n=0,1,\ldots,N}\vert f_1(t_n)\vert \Big)\|g_1\|_0\\
		&\leq K(g_1,h,x_0).
	\end{align*}
	In this way, we have $\|f_1\|_0\leq K(g_1,h,x_0)$.
	
	Take $g_2\in \bar B(g_1,1)$ and set $f_2=F^h(g_2)$. Then $f_2(0)=x_0$ and for any $t\in[t_n,t_{n+1}]$, $n=0,1,\ldots, N-1$,
	\begin{align}\label{f2}
		f_2(t)=f_2(t_n)+b\left((1-\theta)f_2(t_n)+\theta f_2(t_{n+1})\right)(t-t_n)+\sigma(f_2(t_n))(g_2(t)-g_2(t_n)).
	\end{align}
	Denote $e(t):=f_1(t)-f_2(t)$ for any $t\in[0,T]$. It follows from \eqref{f1} and \eqref{f2} that for any $t\in[t_n,t_{n+1}]$, $n=0,1,\ldots, N-1$,
	\begin{align*}
		e(t)&=e(t_n)+\left[b\left((1-\theta)f_1(t_n)+\theta f_1(t_{n+1})\right)-b\left((1-\theta)f_2(t_n)+\theta f_2(t_{n+1})\right)\right](t-t_n)\\
		&\quad+\sigma(f_1(t_n))\left(g_1(t)-g_1(t_n)-(g_2(t)-g_2(t_n))\right)\\
		&\quad+\left(\sigma(f_1(t_n))-\sigma(f_2(t_n))\right)(g_2(t)-g_2(t_n)).
	\end{align*}
	Applying the estimate $\|f_1\|_0\leq K(g_1,h,x_0)$,
	\eqref{lip} and \eqref{linear}, we have 
	\begin{align}\label{et}
		\vert e(t)\vert \leq&\; \vert e(t_n)\vert +hL\left(\vert e(t_n)\vert +\vert e(t_{n+1})\vert \right)\nonumber\\
		&\;+2L\left(1+\|f_1\|_0\right)\|g_1-g_2\|_0+2L\vert e(t_n)\vert \left(1+\|g_1\|_0\right)\nonumber \\
		\leq&\;  K(g_1,h,x_0)\left(\vert e(t_n)\vert +\|g_1-g_2\|_0\right)+hL\vert e(t_{n+1})\vert 
	\end{align}
	for any $t\in[t_n,t_{n+1}],~n=0,1,\ldots,N-1$,	where we have used the fact $\|g_2\|_0\leq \|g_1\|_0+1$ for any $g_2\in\bar B(g_1,1)$. By \eqref{et} and $h\leq\frac{1}{2L}$,
	\begin{align*}
		\vert e(t_{n+1})\vert \leq K(g_1,h,x_0)\left(\vert e(t_n)\vert +\|g_1-g_2\|_0\right),\quad n=0,1,\ldots,N-1.
	\end{align*}
	Using the iteration argument, one has 
	\begin{align*}
		\vert e(t_n)\vert \leq K^n(g_1,h,x_0)\vert e(0)\vert +\sum_{i=1}^nK^i(g_1,h,x_0)\|g_1-g_2\|_0,\quad n=1,2,\ldots,N.
	\end{align*}
	From the above formula and $e(0)=0$, it follows that
	\begin{align}\label{sec2eq1}
		\sup\limits_{n=0,1,\ldots,N}\vert e(t_n)\vert \leq K(g_1,h,x_0)\|g_1-g_2\|_0.
	\end{align}
	Substituting \eqref{sec2eq1} into \eqref{et} yields $\|e\|_0\leq K(g_1,h,x_0)\|g_1-g_2\|_0$, which immediately leads to $\lim\limits_{g_2\to g_1}\|F^h(g_2)-F^h(g_1)\|_0=\lim\limits_{g_2\to g_1}\|e\|_0=0$. This shows that for given $h\leq\frac{1}{2L}$, $F^h$ is continuous. 
	
	Denote $W_{\epsilon}(t)=\sqrt{\epsilon}W_t$, $t\in[0,T]$. Then   $\{W_{\epsilon}\}_{\epsilon>0}$ obeys an LDP on $\mbf C_0([0,T],\mbb R^m)$ (see e.g., \cite[Theorem 5.2.3]{Dembo}) with   
	the good rate function
	\begin{align*}
		I_w(\phi)=
		\begin{cases}
			\frac{1}{2}\int_{0}^{T}\vert \phi'(t)\vert ^2\ud t,\quad&\phi\in\mbf H^1_0(0,T;\mathbb R^m),\\
			+\infty, \quad&\text{otherwise}.		
		\end{cases}
	\end{align*}
	Noting $\bar X^\epsilon=F^h(\sqrt{\epsilon}W)$, we use \cite[Theorem 4.2.1]{Dembo} and the continuity of $F^h$ to conclude that $\left\{\bar{X}^\epsilon\right\}_{\epsilon>0}$ satisfies the LDP on $\mbf C_{x_0}\left([0,T],\mbb R^d\right)$ with the good rate function $\bar J^h$ given by 
	\begin{align*}
		&\;\bar J^h(\varphi)\\
		=&\;\inf_{\{g\in\mbf C_0([0,T],\mbb R^m):F^h(g)=\varphi\}}I_w(g)\\
		=&\;\inf_{\{g\in\mbf H^1_0(0,T;\mbb R^m):F^h(g)=\varphi\}}\frac{1}{2}\int_{0}^T\vert g'(t)\vert ^2\ud t\\
		=&\;\inf_{\{g\in\mbf H^1_0(0,T;\mbb R^m):\varphi(t)=x_0+\int_{0}^tb((1-\theta)\varphi(\hat{s})+\theta\varphi(\check{s}))\ud s+\int_{0}^t\sigma(\varphi(\hat s))g'(s)\ud s, \,t\in[0,T]\}}\frac{1}{2}\int_{0}^T\vert g'(t)\vert ^2\ud t
	\end{align*}
	for any $\varphi\in\mbf C_{x_0}([0,T],\mbb R^d)$. Since $\mbf H^1_0(0,T;\mbb R^m)$ is isomorphic to $\mbf L^2(0,T;\mbb R^m)$ and $\sigma$ is invertible everywhere,
	we have
	\begin{align*}
		\bar J_h(\varphi):=
		\begin{cases}
			\hat S_{T,h}(\varphi),\quad\varphi\in\mbf H^1_{x_0}(0,T;\mbb R^d),\\
			+\infty,\qquad \varphi\in\mbf C_{x_0}([0,T],\mbb R^d)-\mbf H^1_{x_0}(0,T;\mbb R^d).
		\end{cases}
	\end{align*}	 
	According to the definition of the coordinate map $\xi_T$, $X^\epsilon_N=\bar{X}^\epsilon(T)=\xi_T(\bar{X}^\epsilon)$. Again by \cite[Theorem 4.2.1]{Dembo} and the continuity of $\xi_T$, $\{X^\epsilon_N\}_{\epsilon>0}$ satisfies the LDP on $\mbb R^d$ with the good rate function 
	$$I^h(x)=\inf_{\varphi\in\mbf C_{x_0}([0,T];\mbb R^d),\varphi(T)=x}\bar{J}_h(\varphi)=\inf_{\varphi\in\mbf H^1_{x_0,x}(0,T;\mbb R^d)}\hat{S}_{T,h}(\varphi).$$ 
	Thus the proof is complete.
\end{proof}

Now we can apply Theorem \ref{convergence} to giving the convergence of the LDRF $I^h$ of the stochastic $\theta$-method.
\begin{cor}\label{Ihconvergence}
	The numerical solution $\{X^\epsilon_N\}_{\epsilon>0}$ of the stochastic $\theta$-method converges to $\{X^\epsilon(T)\}_{\epsilon>0}$ in large deviations, in the sense that the LDRF $I^h$ of $\{X^\epsilon_N\}_{\epsilon>0}$ converges pointwise to the LDRF $I$ of $\{X^\epsilon(T)\}_{\epsilon>0}$.  And the  convergence order of $I^h$ is $\frac{1}{2}$. Especially, if $\sigma$ is an invertible constant matrix, the convergence order of $I^h$ is $1$.
\end{cor}

\section{Conclusions and future work}\label{Sec6}
The MAM is usually used to study the small-noise-induced transition for nongradient SDEs with small noise, whose central task is to 
numerically solve minimums and minimizers of F-W action functions. In this work, we give a rigorous convergence analysis for an FDM of the MAM, and obtain the  convergence order of the minimum of the discrete F-W action function $\hat{S}_{T,h}$. In addition, the convergence of minimizer sequences of $\hat{S}_{T,h}$ is also presented.
The main novelty of this work is twofold. 
\begin{itemize}
	\item [(1)]
	We first give the convergence rate of minimums of F-W action functionals discretized by  FDMs theoretically for nonlinear SDEs, which provides a supporting for the effectiveness of  MAMs based on the FDM. 
	
	\item [(2)]  We develop a new approach to analyzing the convergence of MAMs based on the equi-coerciveness and locally uniform convergence of discrete F-W action functionals, which can give the convergence rate of their minimums. This is not shared by the theory of $\Gamma$-convergence that only derives the convergence of minimums of parametric minimization problems.  
\end{itemize}

Concerning the future work, we would like to refer to an alternative	idea to analyze the convergence order of $$\inf\limits_{\varphi\in\mbf H^1_{x_0,x}(0,T;\mbb R^d)}\hat{S}_{T,h}(\varphi)=\inf\limits_{(\psi_1,\psi_2,\ldots,\psi_{N-1})\in\mbb R^{N-1}}S_{T,h}(\psi_1,\psi_2,\ldots,\psi_{N-1}).$$  For the simplicity of notations, we illustrate our idea in the case $d=1$. By the classical variational theory (see e.g., \cite{variation}),  any minimizer $\varphi^*$ of $S_T$  solves the Euler--Lagrange equation in the weak sense.  
If $b$ and $\sigma$ are sufficiently smooth, then the weak solution $\varphi^*$ is also sufficiently smooth and  solves the following boundary value problem
\begin{align}\label{EulerS}
	\begin{cases}
		z'(t)=-\sigma'(\varphi^*(t))\sigma(\varphi^*(t))z^2(t)-b'(\varphi^*(t))z(t),\quad t\in(0,T),\\
		\varphi^*(0)=x_0,~\varphi^*(T)=x,
	\end{cases}
\end{align}
where $z(t):=\sigma^{-1}(\varphi^*(t))((\varphi^*)'(t)-b(\varphi^*(t)))\sigma^{-1}(\varphi^*(t))$, $t\in[0,T]$.

Let $(\psi^*_1,\psi^*_2,\ldots,\psi^*_{N-1})$ be a minimizer of $S_{T,h}$. 
Then it holds that 	\\$\frac{\partial S_{T,h}}{\partial \psi_n}(\psi^*_1,\psi^*_2,\ldots,\psi^*_{N-1})=0$, $n=1,2,\ldots,N-1$. Accordingly, we have
\begin{align}\label{FDM}
	z_{n+1}=&\;z_n-h\sigma'(\psi^*_n)\sigma(\psi^*_n)z_{n+1}^2-\theta hb'((1-\theta)\psi^*_{n-1}+\theta \psi^*_n)z_n\nonumber\\
	&\;-(1-\theta)hb'((1-\theta)\psi^*_{n}+\theta \psi^*_{n+1})z_{n+1},~ n=1,2,\ldots,N-1,
\end{align} 
where $z_{n+1}:=\sigma^{-1}(\psi^*_n)\big(\frac{\psi^*_{n+1}-\psi^*_{n}}{h}-b((1-\theta)\psi^*_n+\theta\psi^*_{n+1})\big)\sigma^{-1}(\psi^*_n)$, $n=0,\ldots,N-1$, with $\psi_0=x_0$ and $\psi_N=x$

The error estimate between   $\inf\limits_{(\psi_1,\psi_2,\ldots,\psi_{N-1})\in\mbb R^{N-1}}S_{T,h}(\psi_1,\psi_2,\ldots,\psi_{N-1})$  and \\$\inf\limits_{\varphi\in\mbf H^1_{x_0,x}(0,T;\mbb R^d)}S_{T}(\varphi)$ now boils down to that between  
$S_{T,h}(\psi^*_1,\psi^*_2,\ldots,\psi^*_{N-1})$ and $S_{T}(\varphi^*)$. Further, it is observed that $(\psi^*_0,\psi^*_1,\ldots,\psi^*_N)$ actually is the numerical solution of the FDM for \eqref{EulerS}.  Thus, in order to obtain the     convergence order of minimum of $S_{T,h}$, one needs to give the error order of
\begin{align}\label{1order}
	\sup\limits_{n=1,\ldots,N-1}(\vert \psi^*_{n}-\varphi^*(t_n)\vert +\vert z_{n+1}-z(t_n)\vert ).
\end{align}
However, the theoretical analysis for the error order of \eqref{1order} is difficult due to the strong non-linearity of the second order differential equation in the boundary value problem \eqref{EulerS}. Our work in
this direction is still in progress.

\bibliographystyle{plain}
\bibliography{mybibfile}

\end{document}